\documentclass[lean]{afm}

\usepackage{multicol}
\usepackage{amsmath}
\usepackage{amscd}
\usepackage{graphicx}
\usepackage{hyperref}
\usepackage{url}
\usepackage{xcolor} 
\usepackage{listings}
\usepackage[utf8]{inputenc}
\usepackage[T1]{fontenc}
\DeclareUnicodeCharacter{02E3}{\hspace{-1.5mm}$^{\times}$}
\usepackage{tikz}

\newcommand{\GL}{\mathrm{GL}}
\newcommand{\SL}{\mathrm{SL}}

\newcommand{\N}{\mathbb{N}}

\newcommand{\Z}{\mathbb{Z}}
\newcommand{\Q}{\mathbb{Q}}

\newcommand{\mc}{\mathcal}

\usepackage{tikz-cd}
\usetikzlibrary{shapes.geometric, positioning, trees}
\usepackage{fontawesome}

\newcommand{\link}[2]{\href{#1}{#2 \faExternalLink}}
\newcommand{\file}[1]{\lstinline{#1}}

\title{Formalising the Bruhat--Tits Tree}

\author[J. Ludwig and C. Merten]{Judith Ludwig and Christian Merten}
\authorinfo[J. Ludwig]{Heidelberg University}{ludwig@mathi.uni-heidelberg.de}

\authorinfo[C. Merten]{Utrecht University}{c.j.merten@uu.nl}

\VOLUME{2}
\YEAR{2026}
\NUMBER{3}
\firstpage{55}
\DOI{https://doi.org/10.46298/afm.15738}
\receiveddate{May 23, 2025}
\finaldate{February 18, 2026}
\accepteddate{April 14, 2026}

\begin{abstract}
In this article we describe the formalisation of the Bruhat--Tits tree---an important tool in modern number theory---in the Lean Theorem Prover. Motivated by the goal of connecting to ongoing research, we apply our formalisation to verify a result about harmonic cochains on the tree.
\end{abstract}


\begin{document}

\section{Introduction}
The Bruhat--Tits tree is a combinatorial object that serves as a powerful tool in number theory and arithmetic geometry. 
In the simplest setting of the field $\Q_p$ of $p$-adic numbers, the Bruhat--Tits tree is constructed from lattices in $\Q_p^2$ and is a regular infinite tree, where each vertex has $p+1$ neighbours. Its geometry is characterised by its simplicity. However its connections to the structure theory and the representation theory of the group $\GL_2(\Q_p)$ are deep and intricate. 

In this article we report on the formalisation of this useful mathematical concept in the Lean Theorem Prover. Our formalisation adds to the ongoing effort of formalising graduate level concepts in pure mathematics and more specifically in number theory. 

For the formalisation of the Bruhat--Tits tree, we pass to a more general setting and replace $\Q_p$ by any discrete valuation field $K$. Before we describe our project in more detail, let us say a few more words on the applications of the Bruhat--Tits tree. For that it is useful to furthermore assume that $K$ is complete and has finite residue field. 
In this setting, the Bruhat--Tits tree is a powerful tool for the study of subgroups of $\mathrm{GL}_2(K)$ (as well as the closely related groups $\mathrm{SL}_2(K)$ and $\mathrm{PGL}_2(K)$) and their homology, see e.g. \cite[Chapter II]{serre}. Prominently, as explained in op.\ cit.\ it can be used to show Ihara's theorem, that every torsion free subgroup of $\SL_2(K)$ is free. The Bruhat--Tits tree is important in other parts of number theory, arithmetic geometry and representation theory, e.g., via its connection to Drinfeld's upper half plane~$\Omega$ (see \cite{dasgupta-teitelbaum}) or through the relation to classical as well as function field modular forms via harmonic cochains (see e.g.\ \cite{teitelbaum1, teitelbaum2}). The Bruhat--Tits tree has also been explored in mathematical physics (see e.g. \cite{zabrodin, physics}).

Bruhat--Tits theory puts the Bruhat--Tits tree into a larger number theoretic context. Keeping the setting of a non-archi-medean local field~$K$ for simplicity, let us remark that $\GL_2/K$ is an example of a connected reductive group. Now much more generally, to any connected reductive group~$G/K$ one can associate a so called Bruhat--Tits building, a contractible topological space that carries the structure of a polysimplicial complex equipped with an action of $G(K)$. Bruhat--Tits buildings are a key tool for the structure theory of reductive groups, in spirit analogous to symmetric spaces of real reductive Lie groups. We refer to the introduction of \cite{kaletha-prasad} for an overview of the theory. The simplest non-trivial example is the building (as in \cite[Section 4.1]{kaletha-prasad}) for $G=\mathrm{GL}_2$, where one recovers the Bruhat--Tits tree. Figure \ref{BTpic} is an illustration of the Bruhat--Tits tree for $\Q_2$, the field of $2$-adic numbers.

\begin{figure}
\vspace{.5cm}
\begin{center}
    
\begin{tikzpicture}[
		grow cyclic,
		level distance=1.35cm,
		level/.style={
			level distance/.expanded=\ifnum#1>1 \tikzleveldistance/1.5\else\tikzleveldistance\fi,
			nodes/.expanded={\ifodd#1 fill\else fill\fi}
		},
		level 1/.style={sibling angle=120},
		level 2/.style={sibling angle=100},
		level 3/.style={sibling angle=75},
		level 4/.style={sibling angle=60},
		level 5/.style={sibling angle=40, edge from parent/.style={draw, dashed}},
		edge from parent/.style={draw, semithick},
		nodes={circle,draw,inner sep=+0pt, minimum size=2pt},
		]
		\path[rotate=30]
		node [circle,draw, fill=black, inner sep=0pt, minimum size=2.5pt]{}
		child foreach \cntI in {1,...,3} {
			node[circle,minimum size=2.5pt]{}
			child foreach \cntII in {1,...,2} { 
				node[circle,minimum size=2pt]{}
				child foreach \cntIII in {1,...,2} {
					node[circle,minimum size=1.55pt]{}
					child foreach \cntIV in {1,...,2} {
						node[circle,minimum size=1.25pt]{}
						child foreach \cntV in {1,...,2} {}
						          	node[circle,minimum size=1pt]{}
						         	child foreach \cntV in {1,...,2} {}
					}
				}
			}
		};
		\end{tikzpicture}		
        \end{center}

        \caption{The Bruhat--Tits tree for $\Q_2$}\label{BTpic}
        \vspace{-.5cm}
\end{figure}
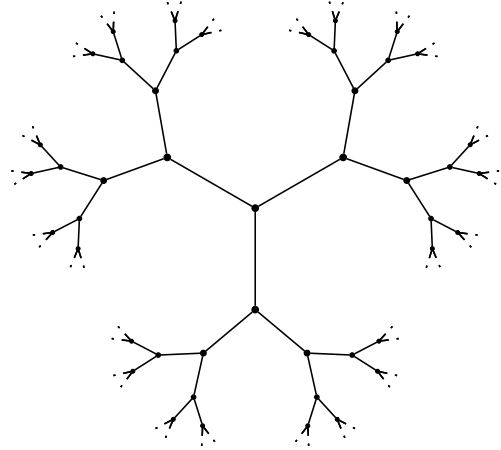		

We review the construction of the Bruhat--Tits tree and describe its formalisation in detail in Section \ref{sec:BT}. For now let us mention that it is constructed as a simple graph, i.e., a graph such that any two vertices are connected by at most one edge, and such that no edge connects a vertex to itself; it has no loops. Moreover, it is connected and has no cycles making it into a so-called tree. As the edges out of any vertex are in bijection with points in the projective line over the residue field, the graph is regular, i.e., the edges out of any vertex have the same cardinality.

The vertices of the Bruhat--Tits tree correspond to certain equivalence classes of lattices in $K^2$. To define the edges of this graph and to understand that it is indeed a tree, it is necessary to relate lattices to each other, i.e., to develop a concept of distance between lattices and to understand chains of lattices. This is achieved by studying how bases of lattices can be transformed into standard shapes and for that one uses a decomposition of the group $\GL_2(K)$ called the \emph{Cartan decomposition}. 

This decomposition exists for $\GL_n(K)$ (and more generally for connected reductive groups) and is of independent interest. In the context of $p$-adic fields for example, it has consequences for the theory of smooth representations of the $p$-adic group $\GL_n(K)$. In Section \ref{sec:cartan} we give the precise statement, more details on context and describe the formalisation of this decomposition.

Our formalisation adds to the ongoing formalisation effort of modern number theory. It is moreover motivated by connecting to current research. An ongoing research project of one of us (J.L.) prominently features the Bruhat--Tits tree via harmonic cochains. These are certain functions on the edges of the tree connecting to representation theory of $\GL_2$ and to rigid analytic functions on Drinfeld's period domain. As we strive to write up our research findings in a way that is error-free, accessible, understandable and unambiguous, we are intrigued by the idea that one day formalisation and verification tools will effortlessly support this process and improve on our documentation. This is still \emph{Zukunftsmusik}, particularly when using the word effortless. But in an attempt to turn up the volume we applied our formalisation of the Bruhat--Tits tree to verify a result about harmonic cochains and gained some mathematical insights along the way. This application is explained in Section~\ref{sec:application} below. We refer to Section \ref{subsec:FOR} for further discussion on the research context. 

To our knowledge, the Bruhat--Tits tree has never been formalised before in a proof assistant. The code for this project is available at the public Github repository
\link{https://github.com/chrisflav/bruhat-tits}{}. We organised our roughly 8000 lines of code into five folders. The folder \file{Graph} contains the main constructions regarding the Bruhat--Tits tree, with the file \file{Graph/Tree.lean} containing the main theorem \lstinline{theorem BTtree} \link{https://github.com/chrisflav/bruhat-tits/blob/b8d0ceceb5cd243b4a4c20be816d591c319e77e9/BruhatTits/Graph/Tree.lean\#L213}{}. We refer to the Readme
\link{https://github.com/chrisflav/bruhat-tits/blob/b8d0ceceb5cd243b4a4c20be816d591c319e77e9/README.md}{} for a detailed overview of the structure of the repository.

We accompany our explanations of the formalisation by some Lean code snippets. We have edited some of them for clarity of exposition. All links to declarations in
mathlib or in our repository are pinned to a specific commit and will therefore remain valid.

\subsection{Lean and mathlib}
This formalisation project was done in the interactive theorem prover Lean 4, which
is also a functional programming language and is based on the Calculus of Inductive Constructions.
For more background on Lean, we refer to \cite{lean4, tpil}. We use
the stable version \lstinline{v4.19.0}.

Our project is based on Lean's monolithic mathematical library
\lstinline{mathlib4}\link{https://github.com/leanprover-community/mathlib4}{}, an extensive
and fast growing open source library of mathematics \cite{mathlib}. At the time of writing
mathlib contains roughly 100 000 definitions and 200 000 theorems. The mathlib community
strives to build a unified library of mathematics. Our project profited from this, as in
particular mathlib contains both the definition of a discrete valuation ring and that of a
regular tree. 
Furthermore we were able to build on the formalisation of foundational undergraduate material
from linear algebra and algebra, importantly, the theory of modules over a ring. Working with
the graph theory part of the library was quite pleasant, as almost all prerequisites
that we needed (such as the notion of a simple graph and a regular tree) were at our disposal. 

Merely in the context of formalising the Cartan decomposition our project would have 
benefited from more existing API on matrices, as this result requires some tedious manipulation
of matrices. 

We have started to integrate parts of the project into mathlib. We aim to fully integrate the Cartan decomposition as well as the formalisation of the Bruhat--Tits tree. For more details see Section \ref{sec:integration}.

\section{The Cartan decomposition of $\mathrm{GL}_n(K)$}\label{sec:cartan}

Let $R$ be a discrete valuation ring, with uniformiser $\varpi$ and fraction field $K$. Let $\mathrm{val}$ denote the valuation on $K$. Note that here and in mathlib by valuation we mean the absolute value (i.e., the norm) and not the additive valuation.
Consider the group $\GL_n(K)$ of invertible $n \times n$ matrices. Let $T(K)$ be the subgroup of diagonal matrices and let $T^{-}$ be the subset of~$T(K) $ given by
\[
T^{-}= \{\mathrm{diag}(\varpi^{f_1}, \ldots , \varpi^{f_n}) \, : \, f_1 \geq \ldots \geq f_n \in \Z \}. 
\]
Then we have
\begin{theorem}[Cartan decomposition]
There is a decomposition of $\GL_n(K)$ into a disjoint union of double cosets
\begin{equation}\label{cartan}
    \GL_n(K) = \bigsqcup_{t \in T^{-}} \GL_n(R) \cdot t \cdot \GL_n(R).
\end{equation}
\label{thm:cartan-decomp}
\end{theorem}

The Cartan decomposition (for $\GL_2(K)$) is important for the construction of the Bruhat--Tits tree (see below). There is an analogous but historically earlier decomposition for Lie groups and Lie algebras with the same name (and which we do not formalise).

The proof of Theorem \ref{thm:cartan-decomp} that we formalise can be described as a Gaussian type algorithm. Alternatively, 
the decomposition can be deduced from the theory of the Smith normal form for matrices over a PID. The advantage of the former is that it is simpler and in parts works over general valuation rings. 

An analogue of the Cartan decomposition exists more generally for connected reductive groups over non-archimedean local fields $K$. For a reference in this generality, see e.g.\ \cite[Proposition 4.4.3]{bruhat-tits}.

When $K$ is a non-archimedean local field, the Cartan decomposition is also an essential result in the theory of smooth representations of $\GL_n(K)$, a theory that is of fundamental importance in the Langlands programme. For instance, the Cartan decomposition implies the countability of the set $\GL_n(K) / \GL_n(R)$ which guarantees that one has Schur's lemma. It is also used to show the commutativity of spherical Hecke algebras and to compute spherical vectors. Note that in these applications, the field is assumed to be complete with respect to the valuation, but for the Cartan decomposition itself this is not necessary. 

One may also ask to what extent the Cartan decomposition holds for  more general valuation rings. There one still has the decomposition 
\[
\GL_n(K) =  \GL_n(R) T(K) \GL_n(R),
\]
where as above $T(K)$ is the abelian subgroup of diagonal matrices. However in the absence of a uniformiser, there is no natural finer  decomposition.

The proof of Theorem \ref{thm:cartan-decomp} is divided into two parts. One first shows
existence of the decomposition, i.e., that one can write any element $g\in \GL_n(K)$ as $g=k_1 t k_2$, for $k_1,k_2 \in \GL_n(R)$ and some $t \in T^{-}$. Then, in a second step, one shows uniqueness of the decreasing tuple $(f_1, \dots , f_n)$ of integers that defines the element $t\in T^{-}$. This implies the disjointness of the double cosets in (\ref{cartan}).
One also shows that the tuple $(f_1, \dots , f_n)$ does not depend on the choice of the uniformiser. 

The proof of existence is an inductive linear algebra argument making use of the fact that one has a valuation to order matrix entries.
The strategy for the formalisation
is inspired by Sébastien Gouëzel's proof
\link{https://github.com/leanprover-community/mathlib4/blob/c44e0c8ee63ca166450922a373c7409c5d26b00b/Mathlib/LinearAlgebra/Matrix/Transvection.lean\#L673}{} that every matrix can be written as a product of a diagonal matrix
and transvections.

To elaborate a bit, note that the group $\GL_n(R)$ contains permutation matrices. So by multiplying a given matrix $g$ from the left and the right by elements in $\GL_n(R)$, we can swap rows and columns and move matrix entries. In particular we can normalise the matrix so that the entry in the lower right corner has maximal valuation. 

\begin{lstlisting}
lemma exists_normalization0 {n : ℕ} (g : Matrix (Fin n ⊕ Unit) (Fin n ⊕ Unit) K) :
    ∃ (k₁ k₂ : GL (Fin n ⊕ Unit) R),
      v ((k₁.val * g * k₂.val) (Sum.inr _) (Sum.inr _)) = g.coeffs_sup v := by
  /- ... -/
\end{lstlisting}

The notation \lstinline{.val} denotes the underlying matrix of an element
of \lstinline{GL (Fin n ⊕ Unit) R} and \lstinline{Fin n} denotes a type of $n$ elements. Note that we use \lstinline{Fin n ⊕ Unit} instead of \lstinline{Fin (n + 1)} to
avoid casting between \lstinline{Fin n} and \lstinline{Fin (n + 1)} when using this
in the induction step. The price we pay is that we need to write
\lstinline{Sum.inr _} for the $(n+1)$-st entry. Moreover \lstinline{g.coeffs_sup v} denotes
the supremum of the valuations \lstinline{v} of the coefficients of the matrix $g$. 

From there we can proceed to eliminate non-zero elements in the last row and column. The outcome is a block matrix in $\GL_{n-1}(K) \times \GL_1(K)$ and we can keep repeating this process to arrive at a diagonal matrix. 
\newpage
\begin{lstlisting}
lemma exists_trafo_isDiag 
  (g : Matrix (Fin n) (Fin n) K) :
  ∃ (k₁ k₂ : GL (Fin n) R),
    IsMonotoneDiag (k₁ * g * k₂) ∧
    (k * g * k₂).coeffs_sup v = g.coeffs_sup v := /- ... -/
\end{lstlisting}

The predicate \lstinline{Matrix.IsMonotoneDiag} \link{https://github.com/chrisflav/bruhat-tits/blob/b8d0ceceb5cd243b4a4c20be816d591c319e77e9/BruhatTits/Cartan/Existence.lean\#L474}{}
expresses that an $n \times n$-matrix
over $K$ is diagonal and that the valuation of the diagonal entries increases monotonically.

Up until this point, the proof works over an arbitrary valuation ring.
The additional assumption that the valuation is discrete is needed for the final step,
namely to get the diagonal matrix into the prescribed form, i.e., with entries ordered
powers of the uniformiser.

\begin{lstlisting}
theorem cartan_decomposition [IsDiscreteValuationRing R]
  (g : GL (Fin n) K) :
  ∃ (k₁ k₂ : GL (Fin n) R) 
    (f : Fin n → ℤ), Antitone f ∧ 
  k₁ * g * k₂ = cartanDiag ϖ hϖ f := /- ... -/
\end{lstlisting}

Here \lstinline{f} takes the role of the integers $f_1,\dots, f_n$ and the condition
$f_1\geq \dots \geq f_n$ is encoded by \lstinline{Antitone f}.
For a discussion of the definition of \lstinline{IsDiscreteValuationRing}
see \cite{local-fields-lean3}.\footnote{The cited article is written for Lean3 and \lstinline{IsDiscreteValuationRing} was still called \lstinline{discrete_valuation_ring}.}

To prove the uniqueness part, one reduces to showing the following: 
Let $k_1,k_2\in \GL_n(R)$ and $t, t' \in T(K)$ with entries integral powers $\varpi^{f_i}$ of the uniformiser such that $t k_1 t' = k_2 \in \GL_n(R)$. Then there exists a permutation $\sigma \in S_n$ such that $\sigma (t)= t'$, where $\sigma(t)$ is the matrix obtained from $t$ by permuting the entries \link{https://github.com/chrisflav/bruhat-tits/blob/b8d0ceceb5cd243b4a4c20be816d591c319e77e9/BruhatTits/Cartan/Uniqueness.lean\#L56C7-L56C39}{}.
For this one argues using the determinant and e.g.\ the fact that $|\det(x)|=1$ 
for any $x\in \GL_n(R)$.

\noindent In Lean, the uniqueness statement then reads as follows.

\begin{lstlisting}
theorem cartan_decomposition_unique 
  {k₁ k₂ k₁' k₂' : GL (Fin n) R}
  {f f' : Fin n → ℤ} (hf : Antitone f)
  (hf' : Antitone f')
  (h : k₁ * cartanDiag ϖ hϖ f * k₂ = k₁' * cartanDiag ϖ hϖ f' * k₂') :
  f = f' := /- ... -/
\end{lstlisting}
For completeness, we also formalise a more literal version of Theorem
\ref{thm:cartan-decomp} \link{https://github.com/chrisflav/bruhat-tits/blob/b8d0ceceb5cd243b4a4c20be816d591c319e77e9/BruhatTits/Cartan/Uniqueness.lean\#L165}{}, but when applying the Cartan decomposition, the two
formulations we give above are more useful.

In the formalisation of the proof we often have to pass between elements of $\GL_n(R)$ and elements of $\GL_n(K)$.
Since these two are distinct types, we use a \emph{coercion} to ease the transition.
\begin{lstlisting}
instance : CoeHead (GL (Fin n) R) (GL (Fin n) K) where
  coe g := GeneralLinearGroup.map R.subtype g
\end{lstlisting}
The function \lstinline{GeneralLinearGroup.map R.subtype} is the
inclusion $\GL_n(R) \to \GL_n(K)$ and registering this coercion means that Lean
automatically inserts this function when it sees an element of $\GL_n(R)$, but expects an
element of~$\GL_n(K)$.

During the formalisation, we expanded the general purpose API for (invertible) matrices
with elementary constructions such as swap matrices and valuations of coefficients.

\section{The Bruhat--Tits Tree} \label{sec:BT}
We keep the notation from the previous section so that $R$ denotes a discrete valuation ring. In this section we describe the formalisation of the Bruhat--Tits tree $\mc{T}$ in detail. 

For that let us briefly review its construction. Consider the set $\mathrm{Latt}(R)$ of $R$-lattices $L\subset K^2$, i.e., of finitely generated $R$-submodules of $K^2$ that span $K^2$ as a $K$-vector space. Two lattices $L, L'$ are called homothetic if one is a scalar multiple of the other, i.e., if there exists $\alpha \in K^{\times}$, such that $L=\alpha L'$. Homothety defines an equivalence relation on $\mathrm{Latt}(R)$ and the set of \emph{vertices of $\mc{T}$} is defined as the set of equivalence classes
\[
V(\mc{T}) := \mathrm{Latt}(R)/{\sim}.
\]
For example the \emph{standard lattice} $R^2$ gives rise to a vertex $v_0$, which we refer to as the standard vertex. 
To define the edges, one associates to each pair of vertices ($v,v')$ a natural number $d(v, v')$, their \emph{distance}. We describe this function along with the formalisation in Section \ref{sec:distance} below.
Its construction requires a normalisation result on bases of pairs of lattices and uses the Cartan decomposition.
 
The distance function allows us to define two vertices $v,v'$ as neighbours, i.e., they are connected by an edge if $d(v,v')=1$. 
An equivalent characterization is as follows: two vertices $v, v' \in \mc{T}$ are neighbours if there exist representatives $v=[L]$ and $v'=[L']$ such that 
\[
\varpi L \subsetneq L' \subsetneq L.
\]
The resulting graph is a tree, i.e., connected and acyclic, and is called the \emph{Bruhat--Tits tree}. When the residue field of $R$ is finite, the Bruhat--Tits tree is regular of degree $q+1$, where $q$ is the cardinality of the residue field. Here regular of degree $d\in \N$ means that every vertex has exactly $d$ neighbours. 

Recall that we are assuming that the field $K$ is discretely valued. As indicated in the introduction, in many applications of the Bruhat--Tits tree the field $K$ is assumed to be a local field and in particular complete (or at least, as e.g.\ in \cite{kaletha-prasad}, Henselian with perfect residue field). For the construction of the Bruhat--Tits tree itself and for further notions and results formalised in this project completeness is not needed and therefore we do not assume it anywhere.

\subsection{Lattices}
For the formalisation of the notion of lattices, let us pass to a more general setting. Let $K$ be any field, let $R\subset K$ be a subring and consider 
$R$-submodules of $K^n$. Our main definition is the predicate
when such a submodule is a lattice:

\begin{lstlisting}
class IsLattice (M : Submodule R (Fin n → K)) : Prop where
  /-- `M` is finitely generated. -/
  isFG : M.FG
  /-- `M` spans `Fin n → K` -/
  spans : Submodule.span K (M : Set (Fin n → K)) = ⊤
\end{lstlisting}

This is implemented as a type class, following the general design principle of mathlib.
While this definition is useful as it is easy to verify, in practice we want to work
with explicit descriptions of lattices via bases. While every free $R$-submodule of $K^n$
of rank $n$ is clearly a lattice the converse might not be true. But when $R$ is a principal ideal domain, every lattice is indeed free.

\begin{lstlisting}
instance IsLattice.free [IsPrincipalIdealRing R] 
    (M : Submodule R (Fin n → K)) 
    [IsLattice M] : Module.Free R M :=
  /- ... -/
\end{lstlisting}

The extensible, typeclass based setup of \lstinline{IsLattice} is well suited
for developing the general theory of lattices.
In order to construct the Bruhat--Tits tree for a given $R$ though,
a \emph{type} of vertices is required that can be equipped with a graph structure. Hence we also need a type of $R$-lattices, which we define as follows. 
\begin{lstlisting}
structure Lattice where
  M : Submodule R (Fin 2 → K)
  isLattice : IsLattice M
\end{lstlisting}
We restrict to $n=2$ mostly for simplicity, since we would have to write
\lstinline{Lattice n R} otherwise.\footnote{For planned adaptations for the integration into mathlib, we refer to Section~\ref{sec:integration}.}
Note that this is unrelated to the order-theoretic notion of a lattice,
which also carries the name \lstinline{Lattice} in mathlib
\link{https://github.com/leanprover-community/mathlib4/blob/c44e0c8ee63ca166450922a373c7409c5d26b00b/Mathlib/Order/Lattice.lean\#L464}{}.
Our \lstinline{Lattice} declaration is therefore namespaced as \lstinline{BruhatTits.Lattice},
but we drop the namespace in the following, when there is no risk of confusion.

The vertices of $\mathcal{T}$ are defined as lattices up to homothety and with the above, the type of vertices can then be defined as a quotient of the type \lstinline{Lattice}.

Note that, when building on top of this, for extensibility as much as possible should
be developed in terms of \lstinline{IsLattice} and only if needed, the constructions
should be applied to the type \lstinline{Lattice} in a last step. An example of this pattern
is the \file{Lattice/Construction.lean} file \link{https://github.com/chrisflav/bruhat-tits/blob/b8d0ceceb5cd243b4a4c20be816d591c319e77e9/BruhatTits/Lattice/Construction.lean\#L10}{}.

\subsection{Distance function}\label{sec:distance}
We return to the setting of a discrete valuation ring $R$ and from now on only consider lattices in $K^2$. An essential tool for defining the Bruhat--Tits tree as a graph and for proving that it is indeed a tree, is the construction of the above mentioned distance function on pairs of vertices. For that, a crucial ingredient is the
following proposition:

\begin{proposition}[\link{https://github.com/chrisflav/bruhat-tits/blob/b8d0ceceb5cd243b4a4c20be816d591c319e77e9/BruhatTits/Lattice/Distance.lean\#L113}{}]
    Let $M$, $L$ be $R$-lattices. Then there exists an ordered $R$-basis $(e, f)$ of $M$ and
    integers $n \ge m$ such that $(\varpi ^ n e, \varpi ^ m f)$ form an $R$-basis of $L$. The
    integers are uniquely determined by $M$ and $L$ and in particular do not depend on
    the $R$-basis of $M$.
    \label{prop:inv-factor}
\end{proposition}

\noindent In Lean the existence part of this proposition reads as:

\begin{lstlisting}
lemma exists_normal_basis (M L : Lattice R) :
    ∃ (bM : Basis (Fin 2) R M.M) (bL : Basis (Fin 2) R L.M) (f : Fin 2 → ℤ),
    Antitone f ∧ 
    ∀ i, (bL i).val = (ϖ ^ f i : K) • (bM i).val
\end{lstlisting}
Here, for a term \lstinline{x : M}\,, the notation \lstinline{x.val} denotes the underlying element
of $K ^ 2$. This is an invocation of the inclusion map $M \to K ^ 2$.
\begin{proof}[Proof of Proposition \ref{prop:inv-factor}]
We apply the Cartan decomposition: First note
that for a lattice $M$ the group $\GL_2(R)$ acts on the set of bases of $M$ on the right
\link{https://github.com/chrisflav/bruhat-tits/blob/b8d0ceceb5cd243b4a4c20be816d591c319e77e9/BruhatTits/Lattice/Distance.lean\#L51}{}.
Now let $(e, f)$ and $(e', f')$ be arbitrary $R$-bases of $M$ and $L$ respectively. Interpreting
these as the columns of elements $g$ and $h$ of $\GL_2(K)$, we may apply the Cartan decomposition
to the product $g^{-1}h$ to obtain a factorization $g^{-1}h = k_1 d k_2$ where
$d = diag(\varpi ^n, \varpi ^ m)$ with $n \ge m$. Then $(e, f) \cdot k_1$ and
$(e', f') \cdot k_2 ^ {-1}$ work with the pair $(n, m)$.
\end{proof}

We can now define the \emph{distance} of two lattices $M$ and $L$, denoted by $d(M, L)$
as the difference $n - m$ of the unique pair of integers from Proposition \ref{prop:inv-factor}. One then easily checks
that $d(M, L) = d(L, M)$, $d(M, L) = 0$ if and only if $M = L$ and that~$d$ is invariant under
scalar multiplication by $K^\times$. 
In particular, $d$ also induces a well-defined
function on the vertices (called \lstinline{inv} in the formalisation)
and we say two vertices $x$ and $y$ are \emph{neighbours} if
$d(x, y) = 1$.

By the properties of the distance, we immediately obtain that this endows $V(\mathcal{T})$ with
the structure of a simple graph, which in Lean reads as follows:
\begin{lstlisting}
def BTgraph : SimpleGraph (Vertices R) where
  Adj L M := BruhatTits.IsNeighbour L M
  symm L M := /- ... -/
  loopless L := /- ... -/
\end{lstlisting}
A priori, $\mathcal{T}$ now has two notions of distance: The one defined above and the
graph-theoretic one, given by infima over lengths of paths. These two notions agree, as we show in Section \ref{sec:acyclicity}.

In order to show that \lstinline{BTgraph} is indeed a tree, we need to show that it
is connected and acyclic. Connectedness is straightforwardly proven by
induction on the distance \linebreak (\lstinline{BTgraph_connected} \link{https://github.com/chrisflav/bruhat-tits/blob/b8d0ceceb5cd243b4a4c20be816d591c319e77e9/BruhatTits/Graph/Graph.lean\#L61}{}).
Acyclicity is more involved and we explain the proof in Section \ref{sec:acyclicity}.

\subsection{Perspectives on lattices}

When working with the Bruhat--Tits graph in practice and while showing it is a tree, we rely heavily on
the fact that any lattice over a PID is free. In particular, many proofs start by choosing appropriate
bases for the lattices in play. How we represent such a basis in Lean, depends on our perspective on lattices.

Restricting to the rank $2$ case, we can view a lattice as

\begin{enumerate}
    \item a free $R$-submodule of $K^2$ of rank $2$,
    \item the $R$-span of a $K$-basis of $K^2$, or
    \item the $R$-span of the columns of an element of $\GL_2(K)$ (as above in the proof of Proposition \ref{prop:inv-factor}).
\end{enumerate}

While informally we freely switch between these perspectives, formally all of these are different and this difference shows most clearly when multiple lattices are in play. We focus on comparing the first two
viewpoints.

For the comparison, we consider a variant of Proposition~\ref{prop:inv-factor}, which
is essentially an unfolding of the definition of~\lstinline{dist}:
\begin{lemma}
Let $M$, $L$ be lattices. Then there exists an $R$-basis $(e, f)$ of $M$ and integers $n \ge m$ such that
$(\varpi ^ n e, \varpi ^ m f)$ is an $R$-basis of $L$ and $n - m = d(M, L)$.
\label{lemma:dist-aux}
\end{lemma}
This is written from the first perspective, as we work with $R$-bases of $M$ and $L$. One
possible formalisation of this statement is the following:

\begin{lstlisting}
lemma exists_repr_dist (M L : Lattice R) :
  ∃ (bM : Basis (Fin 2) R M) (bL : Basis (Fin 2) R L) (f : Fin 2 → ℤ),
    Antitone f ∧
    (∀ i, (bL i).val = (ϖ ^ f i : K) • (bM i).val) ∧
    f 0 - f 1 = dist M L :=
  /- ... -/
\end{lstlisting}

What corresponds to the statement
\emph{$(\varpi ^ n e , \varpi ^ m f)$ is an $R$-basis of $L$} in \ref{lemma:dist-aux}
is expressed in the formal version as
\emph{there exists a basis $(e', f')$ of $L$ such that $e' = \varpi ^ n e$ and $f' = \varpi ^ m f$}. Note though,
that the last two equalities do not type-check.
Formally, the elements \lstinline{ϖ ^ n • e} and \lstinline{ϖ ^ n • f} are of type
\lstinline{M}, while \lstinline{e'} and \lstinline{f'} are of type \lstinline{L}\,.
Hence to state these equalities in Lean, we have to compare
\lstinline{ϖ ^ n • e} and~\lstinline{e'} inside $K^2$, which explains the appearance of
\lstinline{.val} and the explicit type annotation in the above snippet.
In particular this means that we cannot use the induction principle of equality
on these equalities.

In order to avoid casting elements of $M$ into elements of~$K^2$ via \lstinline{.val}
while still keeping the first perspective, one can introduce a predicate
\lstinline{IsBasisOn}, which expresses that
a given indexed family $(b_i)_{i \in I}$ of elements of a vector space $V$ is a basis
for the $R$-submodule $M$, i.e., it is linearly independent over $R$ and generates $M$
as an $R$-module.
Concretely, this could be implemented as \link{https://github.com/chrisflav/bruhat-tits/blob/b8d0ceceb5cd243b4a4c20be816d591c319e77e9/BruhatTits/Utils/IsBasisOn.lean\#L32}{}:

\begin{lstlisting}
class IsBasisOn {R M : Type*} [Ring R] [AddCommGroup V] [Module R V]
    (b : ι → V) (M : Submodule R V) : Prop where
  linearIndependent : LinearIndependent R b
  span_range_eq : Submodule.span R (Set.range b) = M
\end{lstlisting}

Then instead of working with terms of type \lstinline{Basis ι R M}, we only work with families
\lstinline{(b : ι → V) (M : Submodule R V) [IsBasisOn b M]}. 

Our above lemma could 
then be formulated as
\newpage
\begin{lstlisting}
lemma exists_repr_dist₂ (M L : Lattice R) :
  ∃ (b : Fin 2 → (Fin 2 → K))
    [IsBasisOn b M] (f : Fin 2 → ℤ),
    Antitone f ∧
    IsBasisOn (fun i ↦(ϖ ^ f i : K) • b i) L ∧
    f 0 - f 1 = dist M L := /- ... -/
\end{lstlisting}
The \lstinline{.val} has disappeared, since for \lstinline{i : Fin 2}\,, both \lstinline{b i} and
\lstinline{(ϖ ^ f i : K) • b i} are terms of type \lstinline{Fin 2 → K}\,. As a side effect,
the statement is very close to the informal formulation of Lemma \ref{lemma:dist-aux}.

Let us now give a formal statement of Lemma \ref{lemma:dist-aux} from the second perspective. This
means we don't start from $R$-submodules of $K^2$ and their $R$-bases, but instead from $K$-bases of $K^2$ and consider the submodules they span. The key definition is
the following, which associates to a $K$-basis $(b_0, b_1)$ of $K^2$ the $R$-lattice
spanned over $R$ by $(b_0, b_1)$.

\begin{lstlisting}
def Basis.toLattice (b : Basis (Fin 2) K (Fin 2 → K)) : Lattice R where
  M := Submodule.span R (Set.range b)
  isLattice := /- ... -/
\end{lstlisting}

Since in our context, every lattice is free and the underlying elements of an $R$-basis
of a lattice form a $K$-basis of $K^2$, the function \lstinline{Basis.toLattice}
has a section, assigning an arbitrarily chosen basis to a lattice $M$ (\link{https://github.com/chrisflav/bruhat-tits/blob/b8d0ceceb5cd243b4a4c20be816d591c319e77e9/BruhatTits/Lattice/Basic.lean\#L360}{} and
\link{https://github.com/chrisflav/bruhat-tits/blob/b8d0ceceb5cd243b4a4c20be816d591c319e77e9/BruhatTits/Lattice/Construction.lean\#L320}{}).
In particular,
every lattice is of the form \lstinline{Basis.toLattice b} for some basis $b$.

To formulate Lemma \ref{lemma:dist-aux} in this language, it remains to add the twist of a basis by
a tuple of integers: Given a basis $b = (b_0, b_1)$
the twist
\lstinline{b.twist f} by a pair of integers
$f = (f_0, f_1)$ is the basis $(\varpi ^ {f_0} b_0, \varpi ^ {f_1} b_1)$.
With these definitions, the third version of our lemma is:

\begin{lstlisting}
lemma exists_repr_dist₃ (M L : Lattice R) :
  ∃ (b : Basis (Fin 2) K (Fin 2 → K)) (f : Fin 2 → ℤ),
    Antitone f ∧
    M = b.toLattice ∧
    L = (b.twist f).toLattice ∧
    f 0 - f 1 = dist M L := /- ... -/
\end{lstlisting}

Note that instead of comparing two bases (or indexed families), we have two equalities of
lattices: \lstinline{M = b.toLattice} and \lstinline{L = (b.twist f).toLattice}. As
in the second formulation from the first perspective, this eliminates the need for invocations
of \lstinline{.val} and allows to use the induction principle of equality. As an example
of the application of Lemma \ref{lemma:dist-aux} in this formulation, consider the proof that
the distance function is invariant under the action of $\GL_2(K)$:
\begin{lstlisting}
lemma dist_smul_GL_eq_dist (M L : Lattice R) (g : GL (Fin 2) K) :
    dist (g • M) (g • L) = dist M L := by
  obtain ⟨ϖ, hϖ, b, f, hf, hM, hL, hdiff⟩ := exists_repr_dist₃ M L
  subst hM hL
  rw [← Basis.smulGL_toLattice, ← Basis.smulGL_toLattice, Basis.smulGL_twist]
  /- ... -/
\end{lstlisting}

Here in the first line, we invoke \lstinline{exists_repr_dist₃} to obtain
equalities \lstinline{hM} and \lstinline{hL} on which we apply the \lstinline{subst}
tactic to replace all occurrences of \lstinline{M} (resp. \lstinline{L}) by
\lstinline{g • b.toLattice} (resp. \lstinline{g • (b.twist f).toLattice}). As the name suggests, the
\lstinline{subst} tactic is Lean's equivalent of performing substitutions. This is done by applying the induction principle of equality. Then we can use rewrite lemmas to interchange the action of~$\GL_2(K)$ with
various operations on bases and lattices.%
\footnote{See \lstinline{dist_smul_GL_eq_dist}\link{https://github.com/chrisflav/bruhat-tits/blob/b8d0ceceb5cd243b4a4c20be816d591c319e77e9/BruhatTits/Graph/Vertices.lean\#L164}{} for the full proof.}

What makes the second perspective work well is that we may assume that every lattice
is of the form \lstinline{b.toLattice} for some basis \lstinline{b}. To then
do calculations with lattices, we make use of API lemmas like the ones used in the proof of
\lstinline{dist_smul_GL_eq_dist} of which the following is an example,
\begin{lstlisting}
lemma smulGL_toLattice (g : GL (Fin 2) K) (b : Basis (Fin 2) K (Fin 2 → K)) :
    (g • b).toLattice = g • b.toLattice := /- ... -/
\end{lstlisting}
It says that the actions of $\GL_2(K)$ on bases of $K^2$ and on $R$-lattices commute with
the induced lattice construction.

The common pattern in the variants \lstinline{exists_repr_dist₂} and
\lstinline{exists_repr_dist₃} is the avoidance of calculating
in types that depend on the specific lattices in play (such as
\lstinline{Basis (Fin 2) R M} and \lstinline{M} itself). Instead the types
of all data carrying terms only depend on \lstinline{R} and \lstinline{K} and
this is what enables the induction principle of equality.

The difference of the two variants is that the former uses equalities in \lstinline{Fin 2 → K}, while the
latter uses equalities in \lstinline{Lattice R}. Since in the context of this project
we work with vertices of the tree, equality of lattices is typically more useful.
This is why the latter approach is what is now broadly used in the project.

\subsection{Acyclicity}

\label{sec:acyclicity}

The reason why showing connectedness is straightforward is that it only involves
two vertices, which we can immediately bring into a compatible normal form using the lemma
\lstinline{exists_repr_dist₃}. The same trick does not work, when three or more
lattices are in play. However, in good enough cases we may transform into a normal form,
using the natural $\GL_2(K)$-action on $\mathcal{T}$.

By normal form we mean the following: Let $(v_0, \ldots, v_n)$ be a walk in $\mathcal{T}$, i.e., a sequence of adjacent vertices $v_i$. Following \cite[Section 4]{casselman}, we call a
walk $(v_0, \ldots, v_n)$ a \emph{standard chain} if there exists a basis $(e, f)$ of a representative of $v_0$ such that
$(\varpi ^ k e, f)$ is a basis of a representative of $v_k$ for every $0 \le k \le n$.
Note that by Proposition \ref{prop:inv-factor}, any walk of length one is a standard chain.
By construction of the distance, if $(v_0, \ldots, v_n)$ is a standard chain,
then $d(v_0, v_n) = n$.

\noindent The key tool now is:

\begin{lemma}[\link{https://github.com/chrisflav/bruhat-tits/blob/b8d0ceceb5cd243b4a4c20be816d591c319e77e9/BruhatTits/Graph/Tree.lean\#L163}{}]
    Let $p$ be a trail in $\mathcal{T}$, i.e. a sequence of adjacent vertices
    $p = (v_0, \ldots, v_n)$ without repeating edges. Then there exists
    $g \in \GL_2(K)$ such that $(g v_0, \ldots, g v_n)$ is a standard chain.
    \label{lemma:exists-trafo-standard}
\end{lemma}

The proof of Lemma \ref{lemma:exists-trafo-standard} follows closely \cite[Proposition 4.1]{casselman}
and is fairly technical as it involves going back and forth between submodules of the quotient $L/\varpi L$ and submodules of $K^2$ sitting between $\varpi L$ and $L$.
The formalisation of the proof is contained in the file \file{Lattice/Quotient.lean} \link{https://github.com/chrisflav/bruhat-tits/blob/b8d0ceceb5cd243b4a4c20be816d591c319e77e9/BruhatTits/Lattice/Quotient.lean\#L11}{}. This is the slowest of our files but we expect the performance to improve when replacing subrings by algebras (see Section \ref{sec:subrings}). 

Since $\GL_2(K)$ acts by graph isomorphisms on $\mathcal{T}$,  Lemma~\ref{lemma:exists-trafo-standard} implies that the length of
any trail in $\mathcal{T}$ agrees with the distance of its endpoints. So indeed 
our distance function and the graph theoretical distance agree \link{https://github.com/chrisflav/bruhat-tits/blob/b8d0ceceb5cd243b4a4c20be816d591c319e77e9/BruhatTits/Graph/Tree.lean\#L192}{}.

\noindent Moreover, in the special case of distance zero, any trail starting and ending at the same point has length zero, so $\mathcal{T}$ is acyclic.
Putting all the pieces together yields:

\begin{lstlisting}
theorem BTtree : SimpleGraph.IsTree (BTgraph R) where
  isConnected := BTconnected
  IsAcyclic := BTacyclic
\end{lstlisting} 

\subsection{Regularity}
We briefly comment on regularity. Let $\mathfrak{m}$ denote the maximal ideal of $R$ and let $k$ denote the residue field $R / \mathfrak{m}$. For every
lattice $L$, the neighbours of the vertex associated to $L$ are in one-to-one
correspondence to one dimensional subspaces of the two-dimensional $k$-vector space
$L / \mathfrak{m} L$, i.e., to the projectivization $\mathbb{P}_k(L / \mathfrak{m} L)$ \link{https://github.com/chrisflav/bruhat-tits/blob/b8d0ceceb5cd243b4a4c20be816d591c319e77e9/BruhatTits/Graph/Regular.lean\#L257}{}.

To see this it is convenient to introduce the notion of standard neighbours%
\footnote{In fact we use this notion already for the proof of Lemma \ref{lemma:exists-trafo-standard}, and it is introduced in the file \emph{Graph/Edges.lean}
\link{https://github.com/chrisflav/bruhat-tits/blob/b8d0ceceb5cd243b4a4c20be816d591c319e77e9/BruhatTits/Graph/Edges.lean\#L12}{}.}: We say
a lattice $M$ is a \emph{standard neighbour} of the lattice $L$ if $\varpi L \subsetneq M \subsetneq L$. Note that this condition is independent of the choice of $\varpi$,
since $\varpi L = \mathfrak{m} L$.
\begin{lstlisting}
structure IsStandardNeighbour (M L : Lattice R) : Prop where
  lt : M.M < L.M
  ϖlt : ∀ (ϖ : R), Irreducible ϖ → ϖ • L.M < M.M
\end{lstlisting}

To avoid a dependence on an additional parameter \lstinline{ϖ}
in \lstinline{IsStandardNeighbour}, we
have the condition $\varpi L \subsetneq M$ for all uniformisers $\varpi$.

If $M$ and $L$ are standard neighbours, the associated vertices of $M$ and $L$ are
neighbours \link{https://github.com/chrisflav/bruhat-tits/blob/b8d0ceceb5cd243b4a4c20be816d591c319e77e9/BruhatTits/Graph/Edges.lean\#L134}{} and conversely,
if a vertex $v$ is a neighbour of the associated vertex of $L$, there exists an unique representative $M$ of $v$ such that $M$ is a standard neighbour of~$L$.
Hence, the neighbours of the vertex associated to $M$ are in one-to-one correspondence to standard neighbours of~$M$ \link{https://github.com/chrisflav/bruhat-tits/blob/b8d0ceceb5cd243b4a4c20be816d591c319e77e9/BruhatTits/Graph/Regular.lean\#L108}{}. But
the latter are exactly the one-dimensional subspaces of~$L / \varpi L = L / \mathfrak{m} L$.

In particular, if $k$ is finite of cardinality $q$, the Bruhat--Tits tree is regular of
degree $\# (\mathbb{P}_k(L / \mathfrak{m} L)) = q + 1$ \link{https://github.com/chrisflav/bruhat-tits/blob/b8d0ceceb5cd243b4a4c20be816d591c319e77e9/BruhatTits/Graph/Regular.lean\#L286}{}.

\subsection{Group actions}\label{subsec:actions}
An important feature of the Bruhat--Tits tree is that it comes with a well-understood
transitive action of $\GL_2(K)$. To formalise this result we first have to introduce
the notion of a group action on a graph. We define this as an action on vertices that
preserves the adjacency relation. 

\begin{lstlisting}
class GraphAction (G : Type*) [Group G] [MulAction G V] (X : SimpleGraph V) where
  smul_adj_smul (g : G) (x y : V) 
  (h : X.Adj x y) : X.Adj (g • x) (g • y)
\end{lstlisting}

Now the group $\GL_2(K)$ acts on lattices in a way that preserves similarity and distance.
As a consequence we get a graph action on $\mathcal{T}$. 
\begin{lstlisting}
instance : GraphAction (GL (Fin 2) K) (BTgraph R) where
  smul_adj_smul g x y := (adj_smul_smul_iff_adj g x y).mpr
\end{lstlisting}
This action is transitive \link{https://github.com/chrisflav/bruhat-tits/blob/b8d0ceceb5cd243b4a4c20be816d591c319e77e9/BruhatTits/Graph/GroupAction.lean\#L88}{}. The stabilizer of the standard vertex is $K^\times \GL_2(R)$ \link{https://github.com/chrisflav/bruhat-tits/blob/b8d0ceceb5cd243b4a4c20be816d591c319e77e9/BruhatTits/Graph/GroupAction.lean\#L233}{}. Let $(e_0,e_1)$ denote the standard basis of the standard lattice $R^2$ and consider the lattice $L_1$ with basis $(e_0,\varpi e_1)$. Then $v_0=[R^2]$ and $v_1:=[L_1]$ are neighbours and the pointwise stabilizer of the edge $e$ connecting $v_0$ to $v_1$ is given by $K^\times I_1$, where $I_1$ is the subgroup of $\GL_2(R)$ that is upper triangular modulo $\varpi$ \link{https://github.com/chrisflav/bruhat-tits/blob/1d61dd553feb2d1ed35c249298304b818765779c/BruhatTits/Graph/GroupAction.lean\#L491}{}.

For the application in Section \ref{sec:application}, we need to use the action of $\SL_2(K)$
on $\mathcal{T}$, which just acts as a subgroup of $\GL_2(K)$. Contrary to the case of
$\GL_2(K)$ this action is no longer transitive. One way to see this is to note that
$\mathcal{T}$ admits a natural orientation coming from a partition of the vertices into even
and odd; which can be interpreted as a direction on the edges. 
A vertex is called even (resp.\ odd) if the determinant of the matrix spanned by a basis has even (resp.\ odd) additive valuation. 
\begin{lstlisting}
def IsEven (L : Lattice R) : Prop := Even L.valuation
\end{lstlisting}
Equivalently $v$ is even (resp.\ odd) if its distance from the standard vertex $v_0$ is even (resp.\ odd).
The action of $\SL_2(K)$ preserves this orientation, in the sense that it respects even and odd vertices.
\begin{lstlisting}
lemma isEven_specialLinearGroup_smul_iff {x : Vertices R}
  (g : SL (Fin 2) K) : IsEven (g • x) ↔ IsEven x := /- ... -/
\end{lstlisting}

\subsection{Subrings or algebras}\label{sec:subrings}

Informally, when thinking about a domain $R$ and its fraction field $K$, we view $R$ as a
subring of $K$. Alternatively, we may view $K$ as an $R$-algebra, where the induced
ring homomorphism $R \to K$ is injective. In Lean, introducing a variable $R$ that is a subring of an
already declared field \lstinline{K} is as simple as writing \lstinline{(R : Subring K)}. In contrast,
the second point of view is spelled as
\lstinline{(R : Type) [CommRing R] [Algebra R K] [FaithfulSMul R K]}. The assumption
\lstinline{FaithfulSMul R K} expresses that the canonical ring homomorphism $R \to K$ is injective.

When starting the project, we chose the naive \lstinline{Subring} approach, mostly for two reasons:
The first is simplicity. Not only does the initial introduction of \lstinline{R}
require only one variable
instead of four, also every $K$-module is automatically an $R$-module. This removes the need
for additional
\lstinline{Module R V} and \lstinline{IsScalarTower R K V} assumptions, that obfuscate definition and
theorem statements.\footnote{Compare the upstreamed definition
\lstinline{Submodule.IsLattice}
\link{https://github.com/leanprover-community/mathlib4/blob/c44e0c8ee63ca166450922a373c7409c5d26b00b/Mathlib/Algebra/Module/Lattice.lean\#L62}{} with our original one \link{https://github.com/chrisflav/bruhat-tits/blob/b8d0ceceb5cd243b4a4c20be816d591c319e77e9/BruhatTits/Lattice/Basic.lean\#L29}{}.}
Secondly, the field $K$ is carried in the type of $R$, which means the type
\lstinline{IsLattice M} is unambiguous. In the algebra setting, we need an additional
type parameter for $K$.
Having the type of $K$ in $M$ also means that we may register instances such as \lstinline{Module.Finite R M} for
any submodule satisfying \lstinline{IsLattice M}.
\footnote{In the \lstinline{Algebra R K} approach, there are technical ways to still
register this instance using an \lstinline{outParam}
\link{https://lean-lang.org/doc/reference/latest//Type-Classes/Instance-Synthesis/\#outParam}{} for \lstinline{K}.}

One disadvantage of the naive approach is the loss of modularity. While going from
\lstinline{Subring K} to \lstinline{Algebra R K} is automatic by typeclass inference,
results formulated using \lstinline{Subring K} can not be easily used when starting from
\lstinline{Algebra R K}. This issue appears less though, when only one field and one subring
are in play, which is the case in this project. Another disadvantage of
\lstinline{Subring K} is the worse performance, which is related to slow
typeclass synthesization due to more information in the type.\footnote{An example of this effect can be seen in this pull request
\link{https://github.com/leanprover-community/mathlib4/pull/12386}{} to mathlib.}

We eventually chose to use the first approach for the sake of a simpler presentation, but plan
to adapt it to the \lstinline{Algebra R K} setting for integration into mathlib.
 
\section{Harmonic cochains}\label{sec:application}
We apply our formalisation of the Bruhat--Tits tree to verify a result involving harmonic cochains. These form an important tool in the theory of automorphic forms over function fields and in Section \ref{subsec:FOR} below, we give more context as to why we chose this application. 

To introduce harmonic cochains recall that $V(\mc{T})$ denotes the set of vertices of the Bruhat--Tits tree $\mc{T}$. We denote the edges of $\mc{T}$ by $E(\mc{T})$.  By assumption an edge $e \in E(\mathcal{T})$ is an unordered pair $\{v, w\}$ of adjacent vertices (or neighbours, as we called them above) $v, w \in V(\mathcal{T})$. For $v \in V(\mathcal{T})$ and $e \in E(\mathcal{T})$ we write $v \in e$ when $v$ is a vertex of the edge $e$.

Recall from Section \ref{subsec:actions}, that we have a partition of the vertices into even and odd ones. Consider the map 
$w_{\mathrm{par}}: V(\mc{T}) \rightarrow \{{\pm 1}\}$ that sends even vertices to $+1$ and odd vertices to $-1$.

Let $A$ be a commutative ring (with $1$) 
and $M$ be an $A$-module. For any set $X$, we let $\mathrm{Maps}(X,M)$ be the $A$-module of maps from $X$ to $M$, with pointwise addition and scalar multiplication.

\begin{definition}
\begin{enumerate}

\item Define the \emph{Bruhat--Tits Laplacian $\Delta$} as the $A$-module homomorphism
    \begin{align*}
        \Delta: \mathrm{Maps}(E(\mc{T}),M) &\to \mathrm{Maps}(V(\mc{T}),M)\\
        f & \mapsto \left(v\mapsto w_{\mathrm{par}}(v) \sum_{\substack{e \in E(\mc{T})\\ v \in e}} f(e) \right).
    \end{align*}
\item Define the $A$-module of $M$-valued \emph{harmonic cochains} as the kernel of $\Delta$ and denote it by $\mathrm{Har}(\mc{T}, M)$. 
\end{enumerate}
\end{definition}

In number theoretic applications the setting of $A=M=\Z$ and $K$ a local field is of particular interest. As the action of~$\mathrm{SL}_2(K)$ on $V(\mc{T})$ respects the parity of the vertices, the map $\Delta$ is $\mathrm{SL}_2(K)$-equivariant. As we will show below, it is also surjective.

Therefore, we have a short exact sequence of $\mathrm{SL}_2(K)$-equivariant maps
\begin{equation}\label{SES}
0 \rightarrow \mathrm{Har}(\mc{T}, \mathbb{Z}) \rightarrow \mathrm{Maps}(E(\mc{T}),\Z) \xrightarrow{\Delta} \mathrm{Maps}(V(\mc{T}),\Z) \rightarrow 0.
\end{equation}

This short exact sequence is well known (see, for example, \cite[Proof of Prop.\ 1.7]{vdPfonctionstheta}). 
In forthcoming work of one of us (J.L.) with Gebhard Böckle and O\u{g}uz Gezm\.{i}\c{s} we explore it in the context of function field automorphic forms. There the sequence serves as a starting point and a key tool to understand the size and structure of the cohomology group $H^1(\Gamma, \mathrm{Har}(\mc{T}, \mathbb{Z}))$ where $\Gamma$ is a function field analogue of the Ihara group $\mathrm{SL}_2\left(\Z \left[1/p\right]\right)$. 

It was in our interest to formalise a direct elementary proof of surjectivity. As we will explain in the next sections, we can define and show surjectivity of the Laplacian in a more general setting. Alongside these details we also explain the formal verification of our proof. 

\subsection{Surjectivity of the Laplacian}

Let $\mathcal{T}$ be a simple graph that is locally finite, i.e., every vertex has only finitely many adjacent vertices. As for the Bruhat--Tits tree we denote the vertices of $\mathcal{T}$ by $V(\mathcal{T})$ and the edges of $\mc{T}$ by $E(\mathcal{T})$. For a vertex $v\in V(\mathcal{T})$, we let
$$E_v := \left\{ e \in E(\mathcal{T})  \mid v \in e \right\} $$
denote the set of edges containing $v$.
For $v \in V(\mathcal{T})$ denote by $\text{deg}(v) = \text{deg}_{\mathcal{T}}(v)$ the
(finite) number of neighbours of $v$ in~$\mathcal{T}$.

As before let $A$ be a commutative ring and fix a map $w\colon V(\mathcal{T}) \to A^{\times}$. We call $w$ a
\emph{weight function} on $\mathcal{T}$. Let~$M$ be an $A$-module. 

\begin{definition}
    Let $f\colon E(\mathcal{T}) \to M$ be a map. Then we define
    \[
        \Delta_w(f)\colon V(\mathcal{T}) \to M,\quad v \mapsto w(v) \sum_{e \in E_v} f(e)
    .\] 
    
    This defines a map
    \[
        \Delta_w \colon \operatorname{Maps}(E(\mathcal{T}), M) \longrightarrow \operatorname{Maps}(V(\mathcal{T}), M),
    \] 
    which we refer to as the \emph{Laplacian of weight $w$} of the graph~$\mathcal{T}$. If there is no risk of confusion, we also write $\Delta$ for $\Delta_w$.
\end{definition}
Note that as $\mathcal{T}$ is locally finite, the Laplacian is well-defined. 
In Lean, the definition reads as: 
\begin{lstlisting}
def laplace (w : V → Aˣ) (f : X.edgeSet → M) (v : V) : M :=
  w v • ∑ e ∈ X.incidenceFinset' v, f e
\end{lstlisting}

Before we state and prove the surjectivity, we need some preparations. Assume that $\mc{T}$ is a tree. Then the set $V(\mathcal{T})$ of vertices is non-empty and for every pair of vertices
$v, w \in V(\mathcal{T})$ there is a unique path from $v$ to $w$.

\noindent The length of this path is called
the \emph{distance}\footnote{As explained above, in the case of the Bruhat--Tits tree this agrees with the distance function from Section \ref{sec:distance}.} of $v$ and $w$ in $\mathcal{T}$ and is denoted by $d(v,w) = d_{\mathcal{T}}(v, w)$.
Let us fix a vertex $v_0$ of $\mathcal{T}$. We consider $v_0$ as the root vertex of $\mathcal{T}$. 
For a vertex $v$ of $\mathcal{T}$, denote by $|v| = |v|_{v_0}$ the distance
$d(v, v_0)$ of $v$ to the root vertex $v_0$.

As $\mathcal{T}$ is a tree, for any edge $e$ of $\mathcal{T}$ there exist
unique vertices $s(e)$, the \emph{source} of $e$, and $t(e)$, the \emph{target} of $e$, such that
$e = \{s(e), t(e)\}$ and
$|s(e)| < |t(e)|$.

Note that for a vertex $v$ of $\mathcal{T}$ with $|v| > 0$, there
exists a unique edge $o_v$ with $t(e) = v$. In other words, $o_v$ is the
unique adjacent edge of $v$ pointing towards $v_0$. 
Let us assume that $\mathrm{deg}(v) \ge 2$ for all $v$. Then there exists at least one adjacent edge $e$ of $v$ with
$e \neq o_v$, and then $s(e) = v$ and $e$ is pointing away from $v_0$. Note
that this does not require $|v| > 0$ so it also holds for $v = v_0$. 

\begin{definition}
For a vertex $v$ we denote by $\mathrm{Out}_v$ the
subset of $E_v$ containing the edges $e$ with $s(e) = v$ and call it the \emph{outward edge cone} of $v$.
\end{definition}

We have
\[
E_v = \begin{cases}
    \mathrm{Out}_v \cup \{ o_v \} & v \neq v_0 \\
    \mathrm{Out}_v & v = v_0
\end{cases}
.\] 
Note that for every $v \in V(\mathcal{T})$ the outward edge cone $\mathrm{Out}_v \neq \emptyset$ is non-empty, since $\mathrm{deg}(v) \ge 2$. Now
choose for every $v \in V(\mathcal{T})$ a distinguished edge $d_v \in \mathrm{Out}_v$.

\begin{proposition}[\link{https://github.com/chrisflav/bruhat-tits/blob/b8d0ceceb5cd243b4a4c20be816d591c319e77e9/BruhatTits/Harmonic/Basic.lean\#L730}{}]
    Suppose $\mathcal{T}$ is a tree such that $2 \le \text{deg}(v) < \infty$ for all $v \in V(\mathcal{T})$. Let $w$ be a weight function on $\mc{T}$. Then the Laplacian
    $\Delta_w$ is surjective.
    \label{prop:laplace-surj}
\end{proposition}
\newpage
\begin{lstlisting}
lemma laplace_surjective 
  {X : SimpleGraph V} (htree : IsTree X) 
  [∀ v, Fintype (X.neighborSet v)]
  (w : V → A ˣ) (hmindeg : ∀ v, 2 ≤ X.degree v) :
  Function.Surjective (X.laplace M w) := /- ... -/
\end{lstlisting}

\begin{proof}
We prove this by constructing for any function $f\colon V(\mathcal{T}) \to M$ a preimage under $\Delta_w$. So let $f$ be such a function. For $n \in \N_0$ denote by
$B_n$ the ball of radius $n$ in~$E(\mathcal{T})$, by which we mean
\[
    B_n = \{ e \in E(\mathcal{T})  \mid  |t(e)| \le n \}
.\] We now inductively construct a function $h_n \colon B_n \to M$ for $n \ge 0$
such that

\begin{enumerate}[label=(\roman*)]
\item $h_{n+1}|_{B_n} = h_n$,
\item $f(v) = w(v) \sum_{e \in E_v} h_{n+1}(e)$ for every $v \in V(\mathcal{T})$ with
        $|v| \le n$.
\end{enumerate}
Note that the second condition is well-defined, since for $v \in V(\mathcal{T})$ with $|v| \le n$ every edge $e$ with $v \in e$ satisfies $|t(e)| \le n + 1$.

For $n = 0$, the set $B_0$ is empty and there is nothing to do.
For $n = 1$, every $e \in B_1$ satisfies $s(e) = v_0$. Hence
$B_1 = E_{v_0} = \mathrm{Out}_{v_0}$. Now define
\[
h_1 \colon B_1 \to M, \  h_1(e) = \begin{cases}
    w(v_0)^{-1} f(v_0) & e = d_{v_0} \\
    0 & e \neq d_{v_0}
\end{cases}
.\] Condition $(i)$ is trivially satisfied since $B_0 = \emptyset$. Moreover
if $v \in V(\mathcal{T})$ satisfies $|v| \le 0$, we have $v = v_0$ and
\begin{eqnarray*}
    w(v_0) \sum_{e \in E_{v_0}} h_1(e) 
    &=&  w(v_0) h_1(d_{v_0}) \\[-10pt]
    &=& w(v_0)w(v_0)^{-1} f(v_0) \\
    &=& f(v_0).  
\end{eqnarray*}

Now suppose $n \ge 2$ and suppose $h_n$ is already constructed.
For $e \in B_n$ let $h_{n+1}(e) = h_n(e)$. For $e \in E(\mathcal{T})$ with
$|t(e)| = n + 1$ let $v = s(e)$. Note that
$e \in \mathrm{Out}_v$. Define
\begin{equation}
h_{n+1}(e) = \begin{cases}
    w(v)^{-1} f(v) - h_n(o_{v}) &
    e = d_{v} \\
    0 & e \neq d_{v}.
\label{eq:auxborder}
\end{cases}
\end{equation}
By construction, condition $(i)$ is satisfied. For condition $(ii)$, let
$v \in V(\mathcal{T})$ with $|v| \le n$. For $|v| \le n - 1$ the induction hypothesis
applies, thus suppose $|v| = n$. Recall that
$E_v = \left\{ o_v \right\} \cup \mathrm{Out}_v$. Hence
\begin{align*}
    w(v) \sum_{e \in E_v} h_{n+1}(e) &=
    w(v) \left( h_{n+1}(o_v) + \sum_{e \in \mathrm{Out}_v} h_{n+1}(e) \right) \\
                                   &= w(v) \left( h_{n+1}(o_v) + h_{n+1}(d_v) \right) \\
                                   &= w(v) \left( h_n(o_v) + w(v)^{-1} f(v) - h_n(o_v) \right) \\
                                   &= f(v).
\end{align*}

Now define $h\colon E(\mathcal{T}) \to M$ by $e \mapsto h_{|t(e)|}(e)$. Using
conditions $(i)$ and $(ii)$ one now checks easily that $\Delta_w(h) = f$.

\end{proof}
\begin{remark} \ 
\begin{itemize}
    \item Note that our assumption that $deg(v)\geq 2$ is needed (at least for all but one vertex) as the simple example of a graph with two vertices and one connecting edge shows.

\item Furthermore surjectivity breaks as soon as we allow loops with an even number of edges, as the following example illustrates.
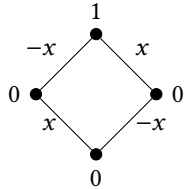
\begin{center}
    \begin{tikzpicture}
    \node[draw, circle, fill=black, inner sep=1.5pt, label=above:$1$] (A) at (0,.8) {};
    \node[draw, circle, fill=black, inner sep=1.5pt, label=left:$0$] (B) at (-.8,0) {};
    \node[draw, circle, fill=black, inner sep=1.5pt, label=right:$0$] (C) at (.8,0) {};
    \node[draw, circle, fill=black, inner sep=1.5pt, label=below:$0$] (D) at (0,-.8) {};
    
    \draw[-] (A) -- (B) node[midway, above left] {\(-x\)};
    \draw[-] (B) -- (D) node[midway, left] {\(x\)};
    \draw[-] (C) -- (D) node[midway, right] {\(-x\)};
    \draw[-] (A) -- (C) node[midway, above right] {\(x\)};
\end{tikzpicture}
\end{center}
Let $v$ denote the top vertex and let $f$ be the function on the vertices that is encoded by the number written next to them, so that e.g. $f(v) = 1$. Then if the edge from $v$ to the middle right vertex is assigned the value $x$, the values at all other edges are predetermined by the zeroes, but then it doesn't work at $v$.
\item In the proof of Proposition \ref{prop:laplace-surj}, one may assume the weight
function $w$ to be trivial, i.e., identically $1$. Indeed, $\Delta_w$ can be recovered from
the Laplacian for the trivial weight function by pre-composing with the automorphism
of
$\mathrm{Maps}(V(\mathcal{T}), M)$ given by $f \mapsto (v \mapsto w(v) f(v))$.
\end{itemize}
    
\end{remark}

\subsection{Verification of surjectivity proof}

The formal verification of the proof of Proposition \ref{prop:laplace-surj} went smoothly and follows
closely the proof outlined above. Remarkably, the LaTeX to Lean code ratio of this proof is
significantly lower than in the rest of the project: Less than 750 lines of code (including
documentation and all needed general purpose graph theory lemmas that were missing in mathlib) for 140 lines of LaTeX. Let us note here a few aspects of the formal proof.

Let \lstinline{V} be a type of vertices and \lstinline{X : SimpleGraph V},
such that every vertex has finitely many neighbours.

Firstly, the construction of the pre-image relies strongly on the arbitrary, but globally
fixed choice of the distinguished edges in the, by assumption nonempty, outward edge cones.
While informally, this might appear as a subtlety,
in Lean this is straightforward:
\begin{lstlisting}
def distinguishedEdge (w : V) 
    (hw : (X.outwardEdgeCone v₀ w).Nonempty) : X.edgeSet :=
  hw.choose
\end{lstlisting}

Secondly, to keep track of the inductive construction of the pre-image and to check the required
equations hold in each step, Lean served its role as a proof assistant very well. In particular,
the construction of $h_n$ from the informal proof can be expressed in a clear way
by a recursive function \lstinline{aux}:

\begin{lstlisting}
def aux (n : ℕ) : X.edgeSet → M := match n with
  /- arbitrary value in step zero. -/
  | 0 => fun _ ↦ 0
  | n + 1 => fun e ↦
      /- if the norm of the edge is `> n + 1`, we give a junk value -/
      if h₁ : n + 1 < X.dist v₀ (X.target v₀ e) then 0
      else
        /- if the norm of the edge is `≤ n`, we use the induction hypothesis -/
        if h₂ : X.dist v₀ (X.target v₀ e) ≤ n then aux n e
        else
          /- if the norm of the edge is `n + 1`, it is on the current border, so we compute -/
          auxBorder w f v₀ n (aux n) e
\end{lstlisting}

The \link{https://github.com/chrisflav/bruhat-tits/blob/b8d0ceceb5cd243b4a4c20be816d591c319e77e9/BruhatTits/Harmonic/Basic.lean\#L469}{\lstinline{auxBorder}} function called in the induction step case is
defined exactly
as explained above in the informal proof, i.e., it is given by the formula (\ref{eq:auxborder}).
Note that, instead of defining \lstinline{aux n}
as a function on the ball of radius $n$, we define it as a function on all edges
of~\lstinline{X}. On edges $e$ with distance $|e| > n$, we instead pick a \emph{junk value}
of $0$. This follows the recurring pattern in formalisation to prefer unconditionally
defined functions and instead only pass the necessary assumptions in proofs.

Finally, the formal counterpart for the definition of the preimage $h$ is given by
\begin{lstlisting}
def preimage (e : X.edgeSet) : M :=
  aux (X.dist v₀ (X.target v₀ e)) e
\end{lstlisting}
and from there it is only a few more lines to check that this is indeed a preimage and complete the verification of Proposition \ref{prop:laplace-surj}.

\section{Concluding remarks}
\subsection{Integration into mathlib}\label{sec:integration}
Our aim is to fully integrate the formalisation of the Cartan decomposition and of the Bruhat--Tits tree into mathlib4. 
We have started to integrate some background and some preparations for the main files (a list of merged pull-requests can be found here \link{https://github.com/leanprover-community/mathlib4/pulls?q=is\%3Apr+is\%3Aclosed+Bruhat-Tits+\%5BMerged+by+Bors\%5D}{}). Notably we have added the notion of $R$-lattices \link{https://github.com/leanprover-community/mathlib4/blob/c44e0c8ee63ca166450922a373c7409c5d26b00b/Mathlib/Algebra/Module/Lattice.lean\#L62}{} and results about the
cardinality of the projectivisation of a finite vector space over a finite field
\link{https://github.com/leanprover-community/mathlib4/blob/c44e0c8ee63ca166450922a373c7409c5d26b00b/Mathlib/LinearAlgebra/Projectivization/Cardinality.lean\#L98}{}.
We do not currently plan to integrate our results on harmonic cochains as they are rather specialised.

For integration into mathlib, besides the above mentioned replacement of
\lstinline{Subring K} by \lstinline{Algebra R K}, the API for lattices, which has been
mostly developed for submodules of $K^2$, needs to be generalised to submodules of
an arbitrary finite dimensional $K$-vector space $V$. We do not expect any major obstactles here.

Regarding the integration of the Cartan decomposition, we are confident that we have worked at
the right level of generality. For example, we work with general valuation rings for as long
as possible. While in our original formalisation we adapt the strategy used in
mathlib's development of transvections
\link{https://github.com/leanprover-community/mathlib4/blob/c44e0c8ee63ca166450922a373c7409c5d26b00b/Mathlib/LinearAlgebra/Matrix/Transvection.lean\#L13}{}, for mathlib we want to generalise as much as possible
from this file to a relative setting of an $R$-algebra $K$, where $K$ is not necessarily
a field, to avoid duplication of the core ideas.

\subsection{Formalising ongoing research} \label{subsec:FOR}

One motivation for this project was to see how much of an ongoing research work can be formalised in the Lean Theorem Prover. The aim of that work is to determine the structure of a group $H^{1}(\Gamma, \mathcal{A}_v^{\times}/\mathbb{C}_v^{\times})$ of so called rigid analytic theta cocycles. Here $\mathcal{A}_v^{\times}$ denotes the group of invertible rigid analytic functions on the Drinfeld upper half plane $\Omega_v$ attached to a local function field $K$, and $\mathbb{C}_v^{\times}$ denotes the subgroup of constant functions. Moreover, $\Gamma$ is a function field analogue of the group $\mathrm{SL}_2(\mathbb{Z}[1/p])$. There is an identification of 
$\mathcal{A}_v^{\times}/\mathbb{C}_v^{\times}$ with the group of harmonic cochains $\mathcal{H}(\mathcal{T},\Z)$ on the Bruhat--Tits tree~$\mathcal{T}$ due to Van der Put \cite[Theorem 2.1]{vdP}. The strategy to prove our structural result is to start from the sequence (\ref{SES}) and analyze the long exact sequence in group cohomology $H^i(\Gamma, -)$ crucially invoking Shapiro's lemma. This leads in particular to a map $H^{1}(\Gamma, \mathcal{A}_v^{\times}/\mathbb{C}_v^{\times}) \rightarrow H^{1}(\Gamma_0(v), \mathbb{Z})$, where~$\Gamma_0(v)$ is a classical congruence subgroup, and provides a link to function field automorphic forms. 

 As described in Section \ref{sec:application}, in the context of this formalisation project we wrote down a direct proof that the sequence (\ref{SES}) is short exact and verified it. One may/should ask whether the argument sketched above can be verified. The answer is: Not without investing significant effort in formalising further theoretical background. For example, to the best of our knowledge, rigid analytic geometry in general and the Drinfeld upper half plane in particular have not been formalised. We also crucially need some background from group cohomology. Here the state of the art is different. By work of Amelia Livingston (\cite{livingston}) the categorical abstract approach is available in mathlib. Furthermore there is work on group cohomology in low degrees \link{https://github.com/leanprover-community/mathlib4/blob/c44e0c8ee63ca166450922a373c7409c5d26b00b/Mathlib/RepresentationTheory/GroupCohomology/LowDegree.lean\#L11}{}. Nevertheless, for our concrete calculations we would have to formalise Shapiro's lemma and develop more API to manipulate our cohomology groups. Given these, it would then be feasible to verify a significant amount of the argument above. 

Let us end by remarking that we ended up using Lean to generalise some initial hypotheses. The reader might have noticed that the spaces of maps in the sequence (\ref{SES}) all have target $\Z$. When we were thinking about a direct proof of surjectivity and about formalising it, we were motivated to generalise the hypothesis to arbitrary commutative rings $A$. After the formalisation of \lstinline{laplace} and the verification of \lstinline{laplace_surjective}, as a map between $A$-valued functions it occurred to us, that instead of $A$-valued maps we should be working more generally with $M$-valued maps, where $M$ is an arbitrary $A$-module. Thanks to Lean, this was as easy as one could hope for; essentially all we had to do was to replace the occurrences of $A$ by $M$ in the code. 

\medskip

\noindent\textbf{Acknowledgments}
We are grateful to the Lean and mathlib community for their work on building a unified digital library of mathematics. Without mathlib, this project would not exist.
J.L.\ would like to thank Gebhard Böckle and O\u{g}uz Gezm\.{i}\c{s} for their support in taking our research on rigid analytic theta cocycles for function fields down the formalisation lane. We would like to thank Gebhard Böckle, Kevin Buzzard, Johan Commelin, O\u{g}uz Gezm\.{i}\c{s}, Peter Schneider and Junyan Xu for helpful comments on an earlier draft of this paper and the anonymous referee for their comments and  suggestions.

\printbibliography

@Book{serre,
 Author = {Serre, Jean-Pierre},
 Title = {Trees. {Transl}. from the {French} by {John} {Stillwell}.},
 Edition = {Corrected 2nd printing of the 1980 original},
 FSeries = {Springer Monographs in Mathematics},
 Series = {Springer Monogr. Math.},
 ISSN = {1439-7382},
 ISBN = {3-540-44237-5},
 Year = {2003},
 Publisher = {Berlin: Springer},
 Language = {English},
 DOI = {10.1007/978-3-642-61856-7},
}

@Book{kaletha-prasad,
 Author = {Kaletha, Tasho and Prasad, Gopal},
 Title = {Bruhat-{Tits} theory. {A} new approach},
 FSeries = {New Mathematical Monographs},
 Series = {New Math. Monogr.},
 Volume = {44},
 ISBN = {978-1-108-83196-3; 978-1-108-93304-9},
 Year = {2023},
 Publisher = {Cambridge: Cambridge University Press},
 Language = {English},
 DOI = {10.1017/9781108933049},
 zbMATH = {7609626},
 Zbl = {1516.20003}
}

@InCollection{dasgupta-teitelbaum,
 Author = {Dasgupta, Samit and Teitelbaum, Jeremy},
 Title = {The {{\(p\)}}-adic upper half plane},
 BookTitle = {\(p\)-adic geometry. Lectures from the 2007 10th Arizona winter school, Tucson, AZ, USA, March 10--14, 2007},
 ISBN = {978-0-8218-4468-7},
 Pages = {65--121},
 Year = {2008},
 Publisher = {Providence, RI: American Mathematical Society (AMS)},
 Language = {English},
 zbMATH = {5355190},
 Zbl = {1153.14021}
}

@Article{teitelbaum1,
 Author = {Teitelbaum, Jeremy T.},
 Title = {The {Poisson} kernel for {Drinfeld} modular curves},
 FJournal = {Journal of the American Mathematical Society},
 Journal = {J. Am. Math. Soc.},
 ISSN = {0894-0347},
 Volume = {4},
 Number = {3},
 Pages = {491--511},
 Year = {1991},
 Language = {English},
 DOI = {10.2307/2939266},
 zbMATH = {20167},
 Zbl = {0735.11025}
}

@article{teitelbaum2,
 author = {Teitelbaum, Jeremy T.},
 title = {Values of {{\(p\)}}-adic {{\(L\)}}-functions and a {{\(p\)}}-adic {Poisson} kernel},
 fjournal = {Inventiones Mathematicae},
 journal = {Invent. Math.},
 issn = {0020-9910},
 volume = {101},
 number = {2},
 pages = {395--410},
 year = {1990},
 language = {English},
 doi = {10.1007/BF01231508},
 url = {https://eudml.org/doc/143809},
 zbMATH = {4208202},
 Zbl = {0731.11065}
}

@Article{bruhat-tits,
 Author = {Bruhat, Fran{\c{c}}ois and Tits, Jacques},
 Title = {Reductive groups over a local field},
 FJournal = {Publications Math{\'e}matiques},
 Journal = {Publ. Math., Inst. Hautes {\'E}tud. Sci.},
 ISSN = {0073-8301},
 Volume = {41},
 Pages = {5--251},
 Year = {1972},
 Language = {French},
 DOI = {10.1007/BF02715544},
 URL = {https://eudml.org/doc/103918},
 zbMATH = {3401076},
 Zbl = {0254.14017}
}

@article{vdP,
 author = {van der Put, Marius},
 title = {Discrete groups, {Mumford} curves and theta functions},
 fjournal = {Annales de la Facult{\'e} des Sciences de Toulouse. Math{\'e}matiques. S{\'e}rie VI},
 journal = {Ann. Fac. Sci. Toulouse, Math. (6)},
 issn = {0240-2963},
 volume = {1},
 number = {3},
 pages = {399--438},
 year = {1992},
 language = {English},
 doi = {10.5802/afst.754},
 url = {https://eudml.org/doc/73309},
 zbMATH = {404402},
 Zbl = {0789.14020}
}

@incollection{livingston,
 author = {Livingston, Amelia},
 title = {Group cohomology in the {Lean} community library},
 booktitle = {14th international conference on interactive theorem proving, ITP 2023, Bia{\l}ystok, Poland, July 31 -- August 4, 2023},
 isbn = {978-3-95977-284-6},
 pages = {17},
 note = {Id/No 22},
 year = {2023},
 publisher = {Wadern: Schloss Dagstuhl -- Leibniz-Zentrum f{\"u}r Informatik},
 language = {English},
 doi = {10.4230/LIPIcs.ITP.2023.22},
 zbMATH = {7949115}
}

@misc{casselman,
  title={The {B}ruhat-{T}its tree of $\mathrm{SL}(2)$},
  author={Bill Casselman},
  note={\url{https://api.semanticscholar.org/CorpusID:1614277}},
  year={2019}
}

@article{zabrodin,
 author = {Zabrodin, A. V.},
 title = {Non-{Archimedean} strings and {Bruhat}-{Tits} trees},
 fjournal = {Communications in Mathematical Physics},
 journal = {Commun. Math. Phys.},
 issn = {0010-3616},
 volume = {123},
 number = {3},
 pages = {463--483},
 year = {1989},
 language = {English},
 doi = {10.1007/BF01238811},
 zbMATH = {4107121},
 Zbl = {0676.22006}
}

@inproceedings{mathlib,
  author    = {{\ignorespaces The} {mathlib Community}},
  title = {The {L}ean mathematical library},
  year = {2020},
  isbn = {9781450370974},
  publisher = {Association for Computing Machinery},
  address = {New York, NY, USA},
  url = {https://doi.org/10.1145/3372885.3373824},
  doi = {10.1145/3372885.3373824},
  booktitle = {Proceedings of the 9th ACM SIGPLAN International Conference on Certified Programs and Proofs},
  pages = {367–381},
  numpages = {15},
  location = {New Orleans, LA, USA},
  series = {CPP 2020}
}

@InProceedings{lean4,
author="Moura, Leonardo de
and Ullrich, Sebastian",
editor="Platzer, Andr{\'e}
and Sutcliffe, Geoff",
title="The {L}ean 4 {T}heorem {P}rover and {P}rogramming {L}anguage",
booktitle="Automated Deduction -- CADE 28",
year="2021",
publisher="Springer International Publishing",
address="Cham",
pages="625--635",
isbn="978-3-030-79876-5",
DOI = {10.1007/978-3-030-79876-5_37},
}

@misc{tpil,
author="Avigad, Jeremy and de Moura, Leonardo and Kong, Soonho and Ullrich, Sebastian",
title="Theorem {P}roving in {L}ean 4",
url = {https://leanprover.github.io/theorem_proving_in_lean4/},
}

@article{physics,
    author = "Chen, Lin and Liu, Xirong and Hung, Ling-Yan",
    title = "{Bending the Bruhat-Tits tree. Part II. The p-adic BTZ black hole and local diffeomorphism on the Bruhat-Tits tree}",
    eprint = "2102.12024",
    archivePrefix = "arXiv",
    primaryClass = "hep-th",
    doi = "10.1007/JHEP09(2021)097",
    journal = "JHEP",
    volume = "09",
    pages = "097",
    year = "2021"
}

@inproceedings{local-fields-lean3, series={CPP ’24},
   title={A {F}ormalization of {C}omplete {D}iscrete {V}aluation {R}ings and {L}ocal {F}ields},
   url={http://dx.doi.org/10.1145/3636501.3636942},
   DOI={10.1145/3636501.3636942},
   booktitle={Proceedings of the 13th ACM SIGPLAN International Conference on Certified Programs and Proofs},
   publisher={ACM},
   author={de Frutos-Fernández, María Inés and Nuccio Mortarino Majno di Capriglio, Filippo Alberto Edoardo},
   year={2024},
   month=jan, pages={190–204},
   collection={CPP ’24} }

@misc{vdPfonctionstheta,
 author = {van der Put, Marius},
 title = {Les fonctions theta d'une courbe de {Mumford}},
 year = {1983},
 language = {French},
 howpublished = {Groupe {Etude} {Anal}. {Ultrametrique}, 9e {Annee}: 1981/82, {No}. 1, {Expose} {No}. 10, 12 p. (1983).},
 url = {https://eudml.org/doc/91869}
}

\end{document}